\documentclass[12pt, letterpaper]{article}

\usepackage{fullpage}
\usepackage{amsfonts}
\usepackage{amsmath, amsthm}
\usepackage{amsthm}
\usepackage{amssymb}
\usepackage{cite}
 \usepackage{graphicx}
 \usepackage[all]{xy}

\newcommand{\F}{\mathbb{F}}
\newcommand{\Z}{\mathbb{Z}}
\newcommand{\C}{\mathbb{C}}
\newcommand{\Q}{\mathbb{Q}}

\newcommand{\pr}{\mathbb{P}}
\newcommand{\p}{\mathfrak{p}}
\newcommand{\q}{\mathfrak{q}}
\newcommand{\I}{\mathfrak{I}}

\newcommand{\Aut}{\operatorname{Aut}}
\newcommand{\GL}{\operatorname{GL}}

\newcommand{\Gal}{\operatorname{Gal}}
\newcommand{\Eva}{\operatorname{G}}

\newcommand{\Hom}{\operatorname{Hom}}
\newcommand{\gon}{\operatorname{gon}}

\newtheorem{thm}{Theorem}
\newtheorem{lemma}{Lemma}
\newtheorem{prop}{Proposition}
\newtheorem{corollary}{Corollary}
\newtheorem{remark}{Remark}

\newtheorem{definition}{Definition}

\begin{document}
\title{On the Gonality of Certain Quotient Varieties}
\author{Jonah Leshin}
\maketitle

\begin{abstract}  
Noether's problem asks whether, for a given field $K$ and finite group
$G$, the fixed field $L:=K(x_h:h \in G)^G$ is a purely transcendental extension of
$K$, where $G$ acts on the $x_h$ by $g \cdot x_h=x_{gh}$. The field
$L$ is naturally the function field for a quotient variety $V:=V(K,G)$. In analogy to the case of curves, we define the
\textit{gonality} of $V$ to be the minimal degree of a dominant
rational map from $V$ to
projective space, which, in a sense, measures the extent to which $L$ may fail to
be purely transcendental over $K$. When $G$ is abelian, we give bounds for the gonality of $V(K,G)$. 
\end{abstract}

\section{Introduction}\label{intro}
The inverse Galois problem for a field $K$ and a finite group $G$ asks
whether there exists a Galois extension $L/K$ with group
$G$. We can embed $G$ in $\GL_n(K)$ for some $n$, so $G$ acts
faithfully on $V=K^n$, and
there is a faithful action of $G$ on the field $K(V)=K(x_1, \ldots,
x_n)$. If the fixed field $K(V)^G$ is a purely
transcendental extension of $K$ (of transcendence degree $n$), then as a consequence of Hilbert Irreducibility, we obtain the existence of a
field (infinitely many fields, in fact) $L/K$ with
$\Eva(L/K)=G$. 
\par A purely transcendental extension $F$ of $K$ is said to be $rational$ over $K$. Given a finite group $G$ and field $K$, consider the regular representation $V_{G}:=\langle x_{g}\rangle_{g \in
  G}$ of $G$ over $K$. Noether's problem asks whether $K(G):=K(V_{G})^G$ is
rational over $K$. An affirmative answer to Noether's problem for $G$
and $K$ implies an affirmative answer to the inverse Galois problem for
$G$ and $K$. Swan was the first to give an example of a group $G$ and
field $K$ for which Noether's problem had a negative answer. He proved
\cite{Swan} that $\Q(\Z/47\Z)$ is not rational over $\Q$ by showing
that Noether's problem for $\Z/p\Z$ was
equivalent to asking whether a prime ideal above $p$ in
$\Z[\zeta_{p-1}]$ is principal, where $\zeta_{p-1}$ is a primitive
$(p-1)$st root of unity. Building on
Swan's work, Lenstra
gave necessary and sufficient conditions under which Noether's problem
has an affirmative answer for all finite abelian groups over any field
\cite{Lenstra}. Saltman pioneered the study of Noether's problem over
the complex numbers, giving an example of a group $G$ of order $p^9$ for
any prime $p\neq 2$ for which $\C(G)$ is not rational \cite{Saltman}. A good deal of
work using techniques ranging from explicit computation to
Galois cohomology and spectral sequences has been done on various
cases of Noether's problem over algebraically closed fields, including
$p$-groups \cite{Peyre1, Kang}, direct products and wreath products \cite{KangRed},
and the alternating groups
$A_n$ \cite{Maeda}; however, there remain many open cases, such as the
rationality of $\C(A_6)$. 

\par Our aim in this paper is to measure the extent to which $K(A)$ may
fail to be rational for an arbitrary finite abelian group $A$ and
field $K$. To this end, we introduce the following quantity.

\begin{definition}
Let $V$ be a variety of dimension $n$ over a field $K$. The
\textit{gonality} of $V$ $\gon(V)$ is the minimal degree of a dominant
rational map $V \dashrightarrow \pr ^n$.
\end{definition}
In the case that $V$ is a curve, $\gon(V)$ is just the standard
gonality of $V$, a quantity that has been studied extensively and
about which many questions remain. The gonality of
hypersurfaces has also been studied \cite{hypergon}.

\par In our case, we fix a finite abelian group $A$ and a field
$K$, and let $V=V_{K,A}$ be the variety (up to birational equivalence) with function field $K(A)$. That is,
$V=\mathbb{A}^{|A|}_{K}/A$. 

\begin{definition}
For a field $K$ and a finite group $G$, define
$\gon(K,G)=\gon(V_{K,G})$. 
\end{definition}
An equivalent definition of $\gon(K,G)$, with which we will work primarily, is
\[ \gon(K,G)= \min_{K\subseteq L\subseteq K(G)}\{[K(G):L]: L/K \textrm{ is rational} \}.
\]
That is, $\gon(K,G)$ is the minimum degree of $K(G)$ over any
field that is rational over $K$. Since our original field
$K(h:h\in G)$ is finitely generated over $K$, $K(G)$ is finitely generated
as well, so a transcendence basis $S$ for $K(G)$ is finite. Since
$K(G)$ is finitely generated and algebraic over $K(S)$, it is
finite over $K(S)$. Therefore, the quantity $\gon(K,G)$ is well
defined. For example,
Swan's result that $\Q(\Z/47\Z)$ is not rational may be written as
$\gon(\Q,\Z/47\Z)\geq 2$. For any finitely generated field $E$ over $F$, we can
similarly define the gonality $\gon(E)$ of $E$ to be the minimum over all
transcendence bases $S$ for $E/F$ of $[E:F(S)]$.

\par  In the case $G=\Z/p\Z$, we may embed $G$ in the
symmetric group $S_p$, which acts naturally on $W:=K^p$ by permuting
coordinates. The algebraic independence of the elementary symmetric
polynomials gives the rationality of $K(W)^{S_p}$. The existence of
this  rational
subfield of $K(G)$ implies $\gon(K,G) \leq \nobreak (p-1)!$ 
\par Our main
theorem is:

\begin{thm}\label{mainupper}
Let $A$ be a finite abelian group and let $K$ be a field such that
$[K(\zeta_s):K]$ is cyclic for every 
prime power $s$ that divides the order of $A$ and is prime to the
characteristic of $K$. Then $\gon(K,A)$ is less
than an explicit quantity, which is given in the statement of
Theorem \ref{GenAb}.
\end{thm}

In addition to Theorem \ref{mainupper}, we discuss conditional lower
bounds for the gonality of a field over which
$K(A)$ is rational.
 \begin{definition}
If $E$ is a transcendental field extension of $F$ and $S$ is a
transcendence basis for $E/F$ such that $[E:F(S)]=\gon(E)$, then we
will call $S$ a \emph{maximal transcendence basis} for $E$ (over
$F$). 
\end{definition}
\par For an abelian group $A$ and an algebraically
closed field of characteristic prime to $|A|$, Noether's problem is
known to have an affirmative answer, due to Fischer
\begin{thm} \textup{\cite{Bryn}}
Let $A$ be an abelian group of exponent $e$ and $K$ a field of
characteristic prime to $e$ that contains
the $e$th roots of unity $\mu_e$. Then $K(A)$ is rational.
\end{thm}
\begin{proof}
Let $|A|=a$. The group $A$ acts on $K(x_1, \ldots, x_a)$. Let
$V=\oplus_{i=1}^a Kx_i$
be the regular representation of $A$. Since $A$ is abelian and $K$
contains $\mu_e$, $V$ can be diagonalized-- i.e., $V$ has a
basis $\{y_1 , \ldots, y_a\}$ such that for any $g \in A$, $g\cdot
y_i=\chi_i(g)y_i$, for a character $\chi_i \in \nolinebreak \hat{A}=\nolinebreak \Hom(A,K^{*})$. Let $M$ be the
multiplicative free abelian group on the $y_i$, and define a group
homomorphism $\psi:M \to \hat{A}$ by sending $y_i \mapsto \chi_i$. The kernel of $\psi$ is a free
abelian group of rank $a$, generated, say, by $\{z_1, \ldots ,
z_a\}$. By construction, each $z_i \in K(A)$. If $f$ is any element of
$K(y_1, \ldots, y_a)^A$ $(=K(A))$, then since $g$ acts by
scalars on each monomial term of $f$, we must have $f \in K(z_1,
\ldots, z_a)$. Therefore, $K(A)=K(z_1, \ldots, z_a)$, and the $z_i$ are
algebraically independent since there are $a$ of them generating a
field of transcendence degree $a$. 
\end{proof}

\par We note that the question of
gonality in Noether's problem leads to several other natural questions. As alluded to earlier, one can ask about the gonality
of an arbitrary variety. Additionally, rather than just considering the rationality of
$K(G)$, we can 
ask whether $K(G)$ satisfies the weaker condition of stable
rationality--that is, whether $K(G)$ becomes rational upon adding
finitely many indeterminates-- or the even weaker
condition of retract
rationality (see \cite{Saltman} or \cite{InvCt} for a definition and
discussion of retract
rationality). These conditions have been studied for the general case
of a quotient variety $V/G$ when $G$ is
any linear algebraic group acting on a vector space $V$ (assuming such
a quotient makes sense). See \cite{InvCt} for a nice survey of this. 
\par This paper is organized as follows. In the next section, we
introduce notation and review
the results of Lenstra \cite{Lenstra} that we will need for Theorem 
\ref{mainupper} (throughout, our presentation of Lenstra's material is
suitably adapted for our purposes). In
Section \ref{bounding}, we begin by modifying Lenstra's method to
obtain Theorem
\ref{mainupper} in the case $K=\Q$, $G=\Z/p\Z$. We work out this
example in detail because the proof of the general
case of Theorem
\ref{mainupper} proceeds similarly to this case, which is less
hampered by notation. We conclude Section
\ref{bounding} by adding the necessary details for the general case of
Theorem \ref{mainupper}. In Section \ref{secbelow}, we investigate a
certain class of rational subfields of $K(A)$ and give conditional
lower bounds for the gonality of these fields. 

\section{Notation and Lenstra's Setup}\label{Lenstra}
Let $K$ be a field, $\pi$ a group of automorphisms of $K$, and $M$ a
$\pi$-module that is a finitely generated
free $\Z$-module with $\Z$-basis $x_1, \ldots, x_m$. We use
multiplication for the group operation of $M$, so elements of $M$ are
monomials in the $x_i$. The group ring
$K[M]$ is then isomorphic to the ring of Laurent polynomials in $m$ variables over
$K$-- that is, $K[x_1^{\pm 1}, \ldots, x_m^{\pm 1}]$, and its quotient field is the rational field $K(M)=K(x_1, \ldots,
x_m)$. The group $\pi$ acts on $K[M]$ by

\[ \left(\Sigma \alpha_i m_i \right)^{\sigma}=\Sigma
\alpha_i^{\sigma}m_i^{\sigma}, \alpha_i \in K, m_i \in M,
\]
which extends to an automorphism of the field $K(M)$. The units of
$K[M]$ are monomials-- that is, $K^*M$. 

\par We introduce a few other pieces of notation. By $\zeta_m$ we denote a
primitive $m$th root of unity. For a field $K$, we have the natural
injection $\Eva(K(\zeta_m)/K) \hookrightarrow (\Z/m\Z)^*$, which allows us to view $\Eva(K(\zeta_m)/K)$ as a subgroup of
$(\Z/m\Z)^*$. We will take the set of divisors of a positive integer
$n$ to be all positive divisors of $n$. The $n$th cyclotomic
polynomial will be denoted $\Phi_n$. Lastly, the function $\phi$
refers to Euler's $\phi$ function. 
\par We begin by studying our gonality question in the case $K=\Q$, $G=\Z/p\Z$, where $p$ is prime. Let $l=\Q(\zeta_p)$ and let $\p$ be a prime ideal of $\Z[\zeta_{p-1}]$ lying above
the rational prime $p$. Then $\p$ is of the form $(p,
\zeta_{p-1}-t)$, where $t$ is an integer that
generates $(\Z/p\Z)^*$ when reduced modulo $p$. Since $\F_p$ has all
($p-1$)st roots of unity, $p$ splits completely in $\Z[\zeta_{p-1}]$. Let
$\pi=\Gal(\Q(\zeta_p)/\Q)$. The additive group of the ideal $\p$ is a free
$\Z$-module of rank $r:=\phi(p-1)$. If $x_1, \ldots, x_r$ is a $\Z$-basis
for $\p$, written multiplicatively, then $\pi$ acts on the monomials
in the $x_i$ and thus on the field $l(\p)$ (acting on $l$ by Galois
automorphisms). 

\par Let $m$ be a divisor of $p-1$ and let $\pi '$ be the quotient group of  $\pi$
of order $m$. The group $\pi'$ can be identified with $\Gal(L/\Q)$ for a subfield $L \subseteq \Q(\zeta_p)$. We have a
ring homomorphism 
\begin{align} \psi_m: \Z[\pi]\to \Z[\pi']\to \Z[\zeta_m], \label{map}
\end{align} where the first map is induced by the natural quotient map
$\pi \twoheadrightarrow \pi '$, and the
second map is defined by sending a generator of $\pi'$ to
$\zeta_m$. This allows us to view any $\Z[\zeta_m]$-module as a
$\Z[\pi]$-module. For any divisor $m$ of $p-1$, following Lenstra \cite{Lenstra}, we define a functor
$F_m$ from the category of $\pi$-modules to the category of torsion-free
$\Z[\zeta_{m}]$-modules by 
\[ F_m(M)=\big(M\otimes _{\pi} \Z[\zeta_{m}]\big)/\{\textrm{additive torsion}\}, 
\]
where we view $\Z[\zeta_{m}]$ as a $\pi$-module via the map
$\psi_m$. 

\par We can make $\Z/p\Z$ into a $\pi$-module by identifying $\pi$ with
$(\Z/p\Z)^*\cong\Aut(\Z/p\Z)$. There exists a unique map $\Z[\pi] \to \Z/p\Z$ of
$\pi$-modules taking $1 \mapsto 1$. Let $J_p$ denote be the kernel of
this map, a free $\Z$-module of rank $p-1$. Lenstra shows:

\begin{prop}\textup{\cite[Proposition 3.6]{Lenstra}} \label{prop Fm}
$F_{m}(J_p)\cong \p$ if $m=p-1$, where $\p$, as above, is an ideal of
$\Z[\zeta_{p-1}]$ above $p$, and 
$F_{m}(J_p)\cong \Z[\zeta_m]$ if $m \neq p-1$. 
\end{prop}

The utility of the functor $F_m$ is demonstrated in the following
theorem of Lenstra, suitably adapted here for our purposes. 
\begin{thm} \textup{\cite[Proposition 2.4]{Lenstra}} \label{L thm}
Let $M$ be a finitely generated, projective $\pi$-module. The fields $l(M)^{\pi}$ and $l(\oplus_{m|p-1}F_{m}(M))^{\pi}$ are
isomorphic, where
$\pi$ acts separately on each direct summand of $\oplus_{m|p-1}F_{m}(M)$ via the maps $\psi_m$ in \eqref{map}.
\end{thm}
\begin{remark}
One checks that $F_{m}$ respects direct sums, a fact we will use hereon
without further reference. 
\end{remark}
We also have:

\begin{prop}\textup{\cite[Proposition 5.3]{Lenstra}}\label{prop Jp}
The field $\Q(\Z/p\Z)$ is isomorphic to a purely
transcendental extension of $l(J_p)^{\pi}$ (of transcendence degree 1).
\end{prop}

We wish to apply Theorem \ref{L thm} to the case
$M=J_p$, so we need to establish that $J_p$ is projective, which we do now:
\begin{prop}\textup{\cite[Proposition 7.1]{SwanInd}} \label{proj}
Let $R$ be Dedekind domain of characteristic zero and $\pi$ a finite
group of order $n$. Let $I$ be an ideal of $R\pi$ such that the ideal
$(R\pi:I)$ of $R\pi$ and the ideal $nR$ of $R$ are comaximal (that
is, there exists $a \in (R\pi:I), b \in nR$ such that $a+b=1$). Then $I$
is a projective $R\pi$-module. 
\end{prop}
We may apply Proposition \ref{proj} in our case since $p \in (\Z[\pi]:J_p)$ and $\pi$ is of
order $p-1$.

\par Putting together Propositions \ref{prop Fm} and \ref{prop Jp} with
Theorem \ref{L thm}, we find that $\Q(\Z/p\Z)$ is isomorphic to a
rational extension of the field $L_{\p}^{\pi}$, where 
\[
L_{\p}:=l\left(\left(\bigoplus_{m|p-1, m \neq p-1}\Z[\zeta_{m}]\right)\oplus \p\right).
\]

From here, we can proceed as Lenstra does in \cite[Theorem
2.6]{Lenstra} to show
that if $\p$ is principal, then $L_p^{\pi}$, and thus $\Q(\Z/p\Z)$, is
rational over $\Q$:
\newline
\par Suppose that $\p$ is principal, so it is a free
$\Z[\zeta_{p-1}]$-module. One checks \cite[Proposition 2.3]{Lenstra} that
$F_m(\Z[\pi])=\Z[\zeta_m]$, for every $m$ dividing $p-1$. Therefore,
aside from the $m=p-1$ summand, the summands of $\oplus_{m|p-1} F_m(\Z[\pi])$
and $\oplus_{m|p-1} F_m(J_p)$ agree, and for the $m=p-1$ summand, since $\p$ is
assumed to be principal, we have $\Z[\zeta_{p-1}] \cong \p$ as
$\Z[\zeta_{p-1}]$-modules. Thus
\[ \oplus_m F_m(J_p) \cong \oplus_m
F_m(\Z[\pi])\]
 as $\Z[\pi]$-modules, and by applying Theorem \ref{L thm} twice, it
follows that $l(J_p)^{\pi}\cong l(\Z[\pi])^{\pi}$. But $\Z[\pi]$ is a 
$\Z[\pi]$-permutation module-- that is, a free $\Z$-module with a $\Z$-basis
that is permuted by $\pi$-- and for any finitely generated $\Z[\pi]$-permutation
module $N$, $l(N)^{\pi}$ is rational over $l^{\pi}$ \cite[Theorem 1.4]{Lenstra}. Therefore, $l(J_p)^{\pi}$ is rational over $l^{\pi}=\Q$, as
desired. 

\section{Bounding the Gonality from Above}\label{bounding}
\subsection{The Case $K=\Q$, $G=\Z/p\Z$}\label{3.1}
Suppose now that the ideal $\p$ is not principal. Lenstra
\cite{Lenstra} shows this
implies that
$\Q(\Z/p\Z)$ is not rational. In this case, we wish to bound $\gon(\Q,
\Z/p\Z)$ from above. Recall from the Introduction that
\[ \gon (\Q, \Z/p\Z)=\textrm{min}\{[\Q(\Z/p\Z):L]: L/\Q \textrm{ is rational} \}.
\]  

Let $\mathfrak{I}$ be a principal ideal of the ring $\Z[\zeta_{p-1}]$ contained in the ideal $\p$. Consider the field 
\[
L_{\I}:=l\left(\left(\bigoplus_{
      m|p-1, m \neq p-1}\Z[\zeta_{m}]\right)\oplus \I \right),
\]
 a subfield of $L_{\p}$. Since $\I$ is a free $\Z[\zeta_{p-1}]$-module
 of rank 1, 
\[ \left(\bigoplus_{m|p-1, m \neq p-1}\Z[\zeta_{m}]\right)\oplus \I
\cong \bigoplus_{m|p-1}F_m\left(\Z[\pi]\right)
\]
as $\Z[\pi]$-modules. Therefore, $L_{\I}^{\pi}$ is isomorphic to
$l(\oplus_{m}F_m(\Z[\pi]))^{\pi}$, which, as shown in the case when
$\p$ was assumed to be principal, is a rational
extension of $l^{\pi}=\Q$.

\begin{figure}[h]
\begin{center}
  \setlength{\unitlength}{2.5cm}
  \caption{}\label{fig}
 \begin{picture}(1,2)(2.7,-.4)
 \put(3.04,0){$L_{\I}^{\pi}$}
    \put(3.04,.6){$L_{\p}^{\pi}$}
    \put(3.04,1.2){$L_{\p}^{\pi}(T)$}
    \put(1.54, .6){$l(J_p)^{\pi}$}   
    \put(2.44, .6){$\cong$}
    \put(1.54, 1.2){$\Q(\Z/p\Z)$}
    \put(2.44,1.2){$\cong$}
    \put(3.94, .6){$L_{\I}^{\pi}(T)$}
    \put(3.04, -.55){$\Q$}
    \put(3.75, .3){\textrm{\footnotesize{rational deg 1}}}
    \put(3.15, -.3){\textrm{\footnotesize{rational}}}
   
    \put(3.15,.3){$d$}
   \put(3.7,1){$d$}
    
    \put(2.0,.9){\textrm{\footnotesize{rational deg 1}}}
    \qbezier(3.08,.21)(3.08,.18)(3.08,.53)
    \qbezier(3.08,.8)(3.08,.8)(3.08,1.15)
   \qbezier(1.58, .8)(1.58,.8)(1.58, 1.15)
    \qbezier(3.28,.1)(3.28,.1)(3.85,.55)
    \qbezier(3.85, .75)(3.85, .75)(3.4, 1.12)
   \qbezier(3.08, -.35)(3.08, -.35)(3.08, -.08)
\end{picture}
\end{center}
\end{figure}

\par Our next task is to give an upper bound for
$[L_{\p}^{\pi}:L_{\I}^{\pi}]:=d$, which will serve as our upper bound for
$\gon(\Q, \Z/p\Z)$ (see Figure \ref{fig}). Since $l(J_p)\cong
L_{\p}^{\pi}$, by Proposition \ref{prop
  Jp}, we can identify $\Q(\Z/p\Z)$ with $L_{\p}^{\pi}(T)$ for an
indeterminate $T$.  Since $\pi$ acts faithfully on $L_{\I}$, from elementary field theory, we have
$[L_{\p}^{\pi}:L_{\I}^{\pi}]=[L_{\p}:L_{\I}]$. Using elementary field theory
and using the fact that if $M$ and $N$ are two free
$\Z$-modules, then $l(M\oplus N)=l(M)(N)$, we find that
$[L_{\p}:L_{\I}]=[l(\p):l(\I)]$. 

\begin{lemma}\label{M}
Let $M$ and $N$ be free $\Z$-modules with $N \subseteq M$ and $|M:N|< \infty$. Let $K$ be a
field. Then $[K(M):K(N)]
=|M:N|$. 
\end{lemma}

\begin{proof}
By induction on $|M:N|$, we may assume that $M/N$ is cyclic of prime
order $p$. The
result will follow from the following two claims.
\begin{itemize}
\item {\textbf{Claim I}} If $\phi$ is an automorphism of $M$, then
  $[K(M):K(N)]=[K(M):K(\phi(N))]$.  \par \textbf{Proof} We may extend $\phi$
  by linearity to an automorphism of $K(M)$. We have
  \\$\phi(K(N))=K(\phi(N))$, and since $\phi$ sends elements of $K(M)$
  linearly independent over $K(N)$ to elements of $K(M)$
  linearly independent over $\phi(K(N))$, we have \\$[K(M):K(N)]\leq
  [K(M):\phi(K(N))]$. We obtain the reverse inequality using
  $\phi^{-1}$. 
 
\item{\textbf{Claim II}} Let $N, N'$ be submodules of $M$ of index $p$ and let
  $\phi:N \to N'$ be an isomorphism. Then $\phi$ may be extended to an
  automorphism $\tilde{\phi}$ of $M$.  \par \textbf{Proof}
 We can do this by picking an element $x \in M$ that lies in neither $N$
nor $N'$ and setting $\tilde{\phi}(x)=x$.
\end{itemize}
These two claims allow us to assume that if $M$ is generated by
elements $x_1, x_2, \ldots , x_r$, then $N$ is generated by $x_1^p,
x_2, \ldots ,
x_r$, and in this case it is clear that $[K(M):K(N)]=p$. 
\end{proof}

From Lemma \ref{M} and the discussion preceding it, we see that
$\gon(\Q ,\Z/p\Z)$ can be bounded above by 
\[\min_{\I \textrm{ principal, } \I \subseteq \p}|\p:\I|.
\]
 Bounding this quantity is our next task.
\par We take $\I=(\zeta_{p-1}-t)$. By \cite{primrt}, for $p$
satisfying $\log_2(p-1)\geq 24 $, there exists a primitive root $t$
modulo $p$ such that 
\begin{equation} |t|\leq \frac{1}{2}p^{1/2}, \label{rtbnd}
\end{equation}
and if $g(p)$ denotes the least primitive root modulo $p$ in absolute
value, then 
\begin{equation} g(p)=O(p^{1/4}). \label{Obnd}
\end{equation}
We have
$|\p:\I|=N(\I)/N(\p)=N(\I)/p$, where
$N=\textrm{Norm}_{\Q(\zeta_{p-1})/\Q}$. Let
$\gamma=\Eva(\Q(\zeta_{p-1})/\Q))$. We have 
\[ 
N(\I)=\left|\prod_{\sigma \in \gamma}(t-\zeta_{p-1}^{\sigma})\right|\leq(t+1)^{\phi(p-1)}\leq \left(\frac{1}{2}p^{1/2}+1\right)^{\frac{p-1}{2}},
 \] where we are using that
 $\phi(p-1)\leq \frac{p-1}{2}$. We thus obtain a numerical version of Theorem
 \ref{mainupper} for $p$ satisfying $\log_2{p-1}\geq 24$:

\[
\gon(\Q, \Z/p\Z) \leq
\frac{1}{p}\left(\frac{1}{2}p^{1/2}+1\right)^{\frac{p-1}{2}}.
\]
 If we use \eqref{Obnd},
we obtain 
\[
\gon(\Q, \Z/p\Z) \leq O(p^{\frac{p-9}{8}}).
\]
We close this subsection with a series of remarks. 

\begin{remark}
The upper bound for $\gon(\Q, \Z/p\Z)$ given by finding the principal
ideal of minimal index inside $\p$ tends to infinity with $p$ since
for any ideal $\I \subseteq \p$, we can write $\p=\I \mathfrak{I}'$ for some
ideal $\mathfrak{I}'$, so that $|\p:\I|=N(\I)/N(\p)=N(\mathfrak{I}')\geq p$ since the norm
of a prime ideal $\q$ lying over a rational prime $q$ is $q^f$, where
$q^f \equiv 1 \pmod{p-1}$. 
\end{remark}

\begin{remark}
As mentioned in the introduction, we also have the weaker bound
\newline $\gon(\Q, \Z/p\Z) \leq (p-1)!$, given by the field $\Q(x_1 ,\ldots,
x_p)^{S_p}$. 
\end{remark}
\begin{remark}The extension $\Q(\Z/p\Z) / \Q(x_1, \ldots, x_p)^{S_p}$ is
not Galois. Let $L=L_{\I}^{\pi}(T)$, where $T$ is an indeterminate, as in
Figure \ref{fig}. This
is the field that gives the bound in Theorem \ref{mainupper}. It is
unknown to the author for which $\p$ the extension $\Q(\Z/p\Z)/L$ might be Galois,
or whether there necessarily exists any field $L'$ for which $\Q(\Z/p\Z)/L'$ is
finite Galois.  
\end{remark}

\subsection{The General Case}\label{generalcase}
We now extend our results about $\gon(\Q, \Z/p\Z)$ to $\gon(K, A)$,
where $K$ is an arbitrary field and $A$ an arbitrary abelian
group. Following Lenstra, if $char k>0$, write 
\[
A=P\oplus B, \]
 where the order of $P$
is a power of $char k$ and the order of $B$ is prime to $char k$, and write 
\begin{equation} B\cong \oplus_{s\in \Omega}\Z/s\Z, \label{decomp}
\end{equation}
where $\Omega$ is the set of prime powers giving the elementary
divisor decomposition of $B$ (with possible repetitions).
Let $e$ be the
exponent of $B$ and put $L=K(\zeta_{e})$, and $\pi=\Eva(L/K)$. From
hereon we make the following assumption:
\[\textrm{For every prime power } s \textrm{ dividing }
|B|, \pi_s:=\Eva(K(\zeta_s)/K) \textrm{ is cyclic.} 
\]
The only situation this precludes is $s=2^n, n\geq 3$ and $char$ $K\neq
2$. 
We have a map $\pi_s \hookrightarrow \Aut(\Z/s\Z)$,
which we can use to make $\Z/s\Z$ into a $\pi_s$-module. In analogy to
the group $J_p$ defined in Section \ref{Lenstra}, we define 

\[
J_s = \ker \psi: \Z[\pi_s] \to \Z/s\Z,
\]
where $\psi$ is the $\Z[\pi_s]$-module map sending $1 \mapsto 1$. In
\cite{Lenstra}, there is a more general version of Proposition \ref{prop Jp}
stating that $K(A)$ is a rational extension of $L(I)^{\pi}$, where 
\[ 
I = \oplus_{s \in \Omega'} J_s, \textrm{ and }\Omega '=\{s \in
\Omega, s \textrm{ is not a power of }2.\}
\]
The group $\pi$ acts on $I$ via the quotient maps $\pi \to
\pi_s$. 
\par For a given prime power $s=l^u$ dividing $|B|$, let
\[
m_s=[K(\zeta_s):K].
\] 
Recall the definition of the functor $F_m$ from Section
\ref{Lenstra}. In the case that $\pi$ is not cyclic, we must modify our
definition (which will agree with Lenstra's in \cite{Lenstra}). For a divisor $m_s'$
of $m_s$, we define the action of $\pi$ on $\Z[\zeta_{m'_s}]$ to be
given via $\Z[\pi]\to \Z[\pi_{m'_s}]\to \Z[\zeta_{m'_s}]$, where
$\pi_{m'_s}$ is the quotient of $\pi_s$ of order $m'_s$ (with this
notation, $\pi_s=\pi_{m_s}$). 
\par We now define, for a $\pi$-module $M$,
\[
F_{m'_s}(M)=M\otimes_{\pi}\Z[\zeta_{m'_s}].
\]

\begin{remark}
We have abused notation since the
$m'_s$ in both $F_{m'_s}$ and $\pi_{m'_s}$ is
actually representing a particular quotient of $\pi$ of order $m'_s$,
rather than just the number $m'_s$. We will maintain this convention throughout.
\end{remark}

Proposition \ref{prop Fm} can be strengthened to
say that for $m'_s$ dividing $m_s$, $F_{m'_s}(J_s)\cong \Z[\zeta_{m'_s}]$ if
$m'_s \neq m_s$, and $F_{m_s}(J_s)$ is of the form $(\zeta_{m_s}-t,
l)$, with $t$ being an integer whose reduction modulo $l$
generates $\Eva (K(\zeta_l)/K)$, where we are viewing $\Eva(K(\zeta_l)/K)$ as
a subgroup of $(\Z/l\Z)^*$. Let $\mathfrak{a}_s=F_{m_s}(J_s)$,
and write $s=l^u$. The polynomial
 $\Phi_{\phi(s)}$ splits into $\phi(l-1)$ distinct irreducible factors over
 $\F_l$. We also know that the ramification degree of $l$ in
 $\Q(\zeta_{l^{u-1}(l-1)}$) is $l^{u-2}(l-1)$. Therefore, the ideal
 $\mathfrak{a}_s=(\zeta_{m_{s}}-t, l)$, which is a prime ideal lying above $l$ in
 $\Z[\zeta_{m_s}]$, has norm $l$. 

\par For odd $s$, we have from \cite[Proposition 3.3]{Lenstra} that $J_s$ is a
projective $\pi$-module. Let $O'$ be a subset of $\Omega'$
consisting of, for each odd prime $p$ dividing $e$, the largest power of
$p$ dividing $e$.
From \cite[Corollary 2.5]{Lenstra}, a corollary to
our Theorem \ref{L thm}, we obtain 
\begin{equation} L(I)^{\pi} \cong 
L(\oplus_{s \in O', m'_s|m_s} F_{m'_s}(I))^{\pi}. \label{I}
\end{equation} 

\begin{remark}
The index of the direct sum on the right hand side of \eqref{I} should technically
be in bijection with all cyclic quotients of $\pi$; however, it follows from
\cite[Proposition 3.6]{Lenstra} that we only need to consider
quotients of $\pi$ corresponding to subfields of $L$ that are
contained in $K(\zeta_{s})$ for prime powers $s$ dividing $e$, and we
only need to consider odd prime powers by Proposition \ref{0Len}
below. \emph{Important Notational Point: }Because
we only wish to count each such quotient of $\pi$ once, indices of the form
$\{m_s'|m_s ,\textrm{ }s \textrm{ in some subset of } \Omega \}$
are understood to include the integer $1$ exactly once across all $s$,
as opposed to including $1$ as a divisor for each $s$.
\end{remark}

Using \cite[Proposition 2.1]{Lenstra} and \cite[Proposition
3.6]{Lenstra}, and the fact that $K(\zeta_a)\cap
K(\zeta_{b})=K$ if $\gcd(a,b)=1$, we obtain the following.

\begin{prop} \label{0Len}
For $s \in O'$, let $N_s$ denote either $J_s$
  or $\pi_{m_s}$. 
\begin{enumerate}
\item If $\gcd(r,s)=1$ and $m_s'\neq 1$, then $F_{m'_s}(N_r)=0$.
\item If $m_s'|m_s$, and $m_s'\neq m_s$, then $F_{m'_s}(N_s)=0$. 
\item $F_1(N_s)=\Z$ for all $s \in \Omega'$. 
\end{enumerate}
\end{prop}
Thus we have
\begin{gather*}
\oplus_{s \in O', m'_s|m_s}
F_{m'_s}(I)=\oplus_{s \in
  O'}\oplus_{m'_s|m_s}F_{m'_s}(\oplus_{r \in \Omega'}J_r)\cong
\end{gather*}
\begin{gather}
\oplus_{r \in
  \Omega'}\big((\oplus_{m'_r| m_r, m'_r \neq m_r}\Z[\zeta_{m'_r}] )\oplus
\mathfrak{a}_r\big). \label{J}
\end{gather}
\newline

\par Note that as a $\Z$-module, $I$ has rank $\sum_{s \in \Omega'}m_s$
and $ \oplus_{m'_s|m_s, s \in O'} F_{m'_s}(I)$ has rank \linebreak{$\sum_{s
  \in \Omega'} \sum_{m'_s|m_s}[\Q(\zeta_{m'_s}):\Q]$}. These ranks are
equal since $m_s=\sum_{m_s'|m_s}[\Q(\zeta_{m'_s}):\Q]$.
\par We may now proceed as we did in Section \ref{Lenstra}, still
working under the assumption that $\pi_s$ is cyclic for every prime
power $s$ dividing $|B|$, to conclude that $K(A)$ is rational over $K$
if the ideal $\mathfrak{a}_s$ is principal for each
$s$ dividing $|B|$. For we have
\newline
\begin{equation}
L(I)^{\pi} \cong L(\oplus_{s \in O', m'_s|m_s}F_{m'_s}(I))^{\pi} \cong
L(\oplus_{s \in O'} (\oplus_{ m'_s|m_s, m'_s\neq m_s} F_{m'_s}(\oplus_{r
  \in \Omega'}\Z[\pi_r]))
\oplus  \mathfrak{a}_s)^{\pi}, \label{a}
\end{equation}
\newline
where the first isomorphism is from \eqref{I} and the second is a
consequence of Proposition \ref{0Len}, line \eqref{J}, and
\cite[Proposition 2.3]{Lenstra}.
\par Suppose $\mathfrak{a}_s$ is principal. Then
$\mathfrak{a}_s\cong \Z[\zeta_{m_s}]\cong F_{m_s}(\Z[\pi_s])$. Consider the $\Z[\pi]$-module $M:=\oplus_{s \in \Omega'}
\Z[\pi_s]$. From \cite[Corollary 2.5]{Lenstra} and \cite[Proposition 2.1]{Lenstra}, we
conclude that $L(M)^{\pi}\cong \nolinebreak L(\oplus_{s \in \Omega',
  m'_s|m_s}\Z[\zeta_{m'_s}])^{\pi}$. Using \eqref{a} and Proposition \ref{0Len}, we find that
$L(I)^{\pi}\cong L(M)^{\pi}$. The $\Z[\pi]$-module $M$ is a
$\pi$-permutation module, from which it follows that $L(M)^{\pi}$, and
thus $L(I)^{\pi}$, is
rational over $l$ \cite[Proposition 1.4]{Lenstra}. At last we obtain:
\begin{thm}[\cite{Lenstra}]\label{genthm}
Using the notation above, assume that $\pi_s$ is cyclic for every
prime power $s$ dividing $|B|$. If $\mathfrak{a}_{s}$ is principal for
all $s \in \Omega'$,
then $K(A)/K$ is rational. 
\end{thm}
\begin{remark}
Lenstra's version of Theorem \ref{genthm} is stronger than
what we have written; what is actually true is that if
$\mathfrak{a}_s^{n_s}$ is principal for all $s \in \Omega'$, then $K(A)/K$
is rational, where $n_s$ is the multiplicity of $s$ in $\Omega'$. This strengthening ultimately comes from the fact that
$\mathfrak{a}_s^{n_s}$ is principal ideal in $\Z[\zeta_{m_s}]$ if and only if
$\mathfrak{a}_s \oplus \cdots \oplus \mathfrak{a}_s$
($n_s$ summands) is a free $\Z[\zeta_{m_s}]$-module \cite{kaplded}. 
\end{remark}
We can now bound $\gon(K,A)$ analogously to the way we bounded
$\gon(\Q, \Z/p\Z)$. The field $K(A)$ is isomorphic to a purely transcendental
extension of $L(I)^{\pi}$, so we can write $K(A)\cong L(I)^{\pi}(T_1,
\ldots, T_k)$, where $T_1, \ldots, T_k$ are indeterminates and $k$
depends on $K$ and $A$. If $\mathfrak{I}_s$ is
a principal ideal contained in $\mathfrak{a}_s$, then it follows
from Lemma
\ref{M} and the discussion of the case $K=\Q$, $A=\Z/p\Z$ that
\begin{gather*}
[L(\oplus_{s \in O'}(\oplus_{m'_s|m_s, m'_s\neq m_s} F_{m'_s}(\oplus_{r
  \in \Omega'}\Z[\pi_r]))
\oplus \mathfrak{a}_s)^{\pi}:\\L(\oplus_{s \in O'}
( \oplus_{m'_s|m_s, m'_s\neq m_s} F_{m'_s}(\oplus_{r \in \Omega'}\Z[\pi_r]))
\oplus \mathfrak{I}_s)^{\pi}]=
\prod_{s \in \Omega'}
|\mathfrak{a}_s:\mathfrak{I}_s|.
\end{gather*}
The following Theorem now follows in analogy with the $K=\Q$,
$A=\Z/p\Z$ case.

\begin{thm} \label{GenAb}
Let $A$ be an abelian group and $K$ a field for which $K(\zeta_s)/K$
is cyclic for every prime power $s$ that divides $|A|$ and is prime to
the characteristic of $K$. Then, using the notation above, 
\[
\gon(K,A) \leq \prod_{s\in
  \Omega'}\min |\mathfrak{a}_s:\mathfrak{I}_s|,
\]
 where the minimum is taken
over
all principal ideals $\I_s \subseteq \mathfrak{a}_s$ for each $s \in \Omega'$. 
\end{thm}

\begin{remark}\label{course}
In the course of proving Theorem \ref{GenAb}, we have shown that
\[
\gon(l(I)^{\pi})\leq  \prod_{s\in
  \Omega'}\min_{\I_s \subseteq \mathfrak{a}_s, \textrm{ }\I_s\emph{ principal}}|\mathfrak{a}_s:\mathfrak{I}_s|.
\]
\end{remark}

\subsection{An Example of Theorem \ref{GenAb}}\label{ex}
As an example of Theorem \ref{GenAb}, we translate the result into a
numerical bound for the case $K=\Q$ (note that if $\gon(\Q,A)=d$, and
$K$ is a number field, then $\gon(K,A)\leq d$ because if $E$ is a
rational field over $\Q$ with $[\Q(A):E]=d$, then $K \otimes E$ is a rational
field  over $K$ with
$[K(A):K\otimes E]=d$). The
Sylow-$l$ subgroups of $A$ for a prime $l$ dividing $|A|$ can be dealt with
independently. Thus we take a prime $l$, and for the remainder
of Section \ref{ex}, we assume that $A$ is an
$l$-group. We write  
\[A=(\Z/l^{u_1}\Z)^{v_1} \oplus \cdots \oplus
 (\Z/l^{u_b}\Z)^{v_b}
\] 
with $u_i < u_{i+1}$, so $l^{u_i}$ appears with
 multiplicity $v_i$ in $\Omega'$. Let $s=l^{u_i}$. We have \linebreak
$m_{l^{u_i}}= \phi(s)=
l^{u_{i-1}}(l-\nolinebreak 1)$.  For each $s \in \Omega'$, we take
 $\I_s=(\zeta_{m_s}-t)$, where $t$ is an integer whose reduction modulo
 $l$ generates $(\Z/l\Z)^*$. 
\par From line \eqref{Obnd} in Section
 \ref{3.1}, we have $t=O(l^{1/4})$. Let $s=l^u$, so
 $\phi(s)=l^{u-1}(l-1)$. Since $\mathfrak{a}_s$ has norm $l$, we find that for each $s$,
 \begin{equation}\label{Ogen}
|\mathfrak{a}_s:\I_s|=N(\I_s)/N(\mathfrak{a}_s)\leq
O\Big(l^{\frac{\phi(\phi(s))-4}{4}}\Big).
\end{equation}

\begin{remark}
Let $h_{n}$ denote the class number of $\Q(\zeta_{n})$. Since
$\mathfrak{a}_s^{h_{\phi(s)}}$ is principal, letting \linebreak $h=h_{\phi(s)}$,
we obtain the bound
\begin{equation}\label{cl}
\min_{\I_s \emph{ principal, }\I_s\subseteq a_s}|\mathfrak{a}_s:\I_s| \leq |\mathfrak{a}_s:\mathfrak{a}_s^{h}|=N(\mathfrak{a}_s)^{h-1}=l^{h-1}.
\end{equation}
Recall that the class number $h_n^+$ of the maximal real subfield of
$\Q(\zeta_n)$ divides $h_n$. Much is known about the quotient
$h_n/h_n^+:=h_n^{-}$, while comparatively little is known about $h_n^+$
(these quantities are discussed in detail in \cite{Wash}, for
example). 
From \cite[Theorem 4.20]{Wash}, we have 
\begin{equation}\label{hn}
h_n^-\sim n^{\frac{1}{4}\phi(n)}.
\end{equation}
Taking $n=l^{u-1}(l-1)$ and using the bound \eqref{hn} in place of $h$
in \eqref{cl} (which is cheating, of course, since \eqref{hn} does not
account for $h^+$) already gives a much larger bound for
$\min_{\I_s \emph{ principal, }\I_s\subseteq a_s}|\mathfrak{a}_s:\I_s|$ than we obtain via \eqref{Ogen}.
\par To the author's knowledge, there is no known asymptotic formula (or
even non-trivial lower bound) for $h_n^+$. 
\end{remark}

\par Thus we find that  $ \prod_{s\in
  \Omega'}\min|\mathfrak{a}_s:\mathfrak{I}_s| $ can be bounded above by
$O(l^C)$, where, by \eqref{Ogen}, we may take $C$ to be
\begin{equation}\label{C}
\frac{1}{4}\left(\sum_{i=1}^{b}v_i\big(l^{u_i-2}(l-1)\phi(l-1)-4)\right).
\end{equation}

 We can also obtain an effective upper bound
for $\min_{\I_s \textrm{principal, } \I_s \subseteq \mathfrak{a}_s} |\mathfrak{a}_s:\I_s|$ for each $s \in \Omega'$, and thus for
$\gon(\Q,A)$, by means of
\eqref{rtbnd} from Section \ref{Obnd}. We obtain
\begin{equation} \label{C'}
\min_{\I_s \textrm{principal, } \I_s \subseteq \mathfrak{a}_s} |\mathfrak{a}_s:\I_s| \leq \frac{1}{l}\left(\frac{1}{2}l^{1/2}+1\right)^{\phi(\phi(s))}.
\end{equation}

\begin{remark}
One checks that the bound for $\gon(\Q, A)$ coming from \eqref{C'}
(and thus from \eqref{C})
beats the bound one gets from the rational field 
$\Q(x_1,\ldots ,x_{|A|})^{S_{|A|}}$, of which $\Q(A)$ is a degree
$(|A|-1)!$ extension. 
\end{remark}

\section{Rational Subfields of K(A)} \label{secbelow}
Let $A$ be an abelian group and $K$ a field, and suppose there is a
rational field $K'$ inside $K(A)$. To simplify the presentation, we will
assume that $char K$ is prime to $|A|$; if \linebreak $\gcd(char K, |A|)\neq
1$, the results of this section are easily modified since $K(A)$ is rational over
$L(I)^{\pi}$ (notation as in Section \ref{generalcase}).  Ideally, we would like a lower
bound for $[K(A):K']$ in terms of $K$ and $A$, which would provide a lower bound for
$\gon(K,A)$. In this section, we investigate certain rational
subfields of $K(A)$, and discuss how various hypthoses lead to
conditional lower bounds for $\gon(K,A)$. 
\par We begin by setting notation. Let $\Omega'=\{s_1, \ldots, s_n\}$ be the
odd elementary divisors of $|A|$, and let $\bar{\Omega}$ be a maximal subset
among all subsets of $\Omega'$ whose elements are distinct. Let $I=\oplus_{s \in
  \Omega'}J_s$, as defined in
Section \ref{generalcase}. For $s \in \Omega'$, set $\pi_s=\Eva(K(\zeta_s)/K)$, with order
$m_s$. Let $e$ be the exponent of $|A|$, let $l=K(\zeta_e)$, and let
$\pi=\Eva(l/K)$.  Define the following sets:
\begin{align*}
 U&=\{x_{1;0},\ldots, x_{1;s_1-1}, \ldots \ldots, x_{n;1}, \ldots, x_{n;s_n-1}\},
 \quad |U|=\sum s_i \\
S&=\{x_{1;0}, \ldots, x_{1;m_{s_1}-1}, \ldots \ldots, x_{n;0}, \ldots,
x_{n;m_{s_n}-1}\}, \quad |S|=\sum m_{s_i}\\
T&=U\setminus S, \quad |T|=|U|-|S|.
\end{align*}
In order for our methods to work,
we need to make the following assumptions, which we do for the
remainder of Section \ref{secbelow}:
\begin{gather}\label{test}
\textrm{For all s }\in \Omega', \textrm{ either
  }m_s=[K(\zeta_{s}):K] \textrm{ is even
  or } \zeta_{s}\in K. \\ \label{test2}
\textrm{For all } s \in \Omega, K(\zeta_s)/K \textrm{ is cyclic. } 
\end{gather}

\begin{figure}[h]
\caption{}\label{diamond1}
\begin{center}
  \setlength{\unitlength}{2.5cm}
  
 \begin{picture}(.3,1.5)(3,.5)
    
    \put(2.75,.6){$K(S)$}
    \put(4.5,1.2){$l(I)^{\pi}$}
    \put(1.7,1.2){$K':=K(S,T)$}
    \put(2.73, 1.8){$l(I)^{\pi}(T)$}
     \put(2.58,1.48){$d$}
      \put(3.75,.8){$d$}
    \put(3.87,1.52){\footnotesize{rational}}
     \put(2.2,.85){\footnotesize{rational}}

    \qbezier(2.9,.8)(2.9,.8)(2.6,1.15)
      \qbezier(3.2,.8)(4.4,1.15)(4.45, 1.15)
    \qbezier(2.6, 1.35)(2.6,1.35)(2.95,1.74)
    \qbezier (4.4, 1.25)(4.4,1.25)(3.17, 1.74)

\end{picture}

\end{center}
\end{figure}
 We take the
subset $S_i:=x_{i;0}, \ldots, x_{i;m_{s_i}-1}$ of $S$ to be a maximal
transcendence basis for $l(J_{s_i})^{\pi}$ over
$K$. We have $d:=[l(I)^{\pi}:K(S)]\leq \gon(l(I)^{\pi})$.  By the
generalization in \cite{Lenstra} of our Proposition
\ref{prop Jp}, $K(A)$ is a rational extension of a subfield isomorphic to
$l(I)^{\pi}$, which we identify with $l(I)^{\pi}$. We may take
$l(I)^{\pi}(T)=K(A)$. In \cite{Lenstra}, Lenstra
establishes that $l(I)^{\pi}$ is rational if and only $K(A)$ is
rational (one direction being obvious). In other words, $\gon(K,A)=\gon(l(I)^{\pi})$  in the case that $\gon(K,
\Z/p\Z)=1$, and we might speculate that $\gon(K,A)=\gon(l(I)^{\pi})$ in
all cases.

\par Let $m_e=[l:K]$. For each $k$, $1\leq k \leq n$, define
\[
z_{k;i}=\sum_{j=0}^{s_k-1}\zeta_{s_k}^{ij}x_{k;j}.
\]
This is the discrete Fourier transform of vector spaces:
$\oplus_{0\leq i \leq s_k-1} lx_{k;i} \to \oplus_{0\leq i \leq s_k-1}
lz_{k;i}$, so for each $k$, $l(x_{k;0}, \ldots, x_{k;s_k-1})=l(z_{k;0}, \ldots, z_{k;s_k-1})$. The
group $\pi$ acts, via its quotient $\pi_{s_k}$, on $l(z_{k;0}, \ldots,
z_{k:s_k-1})$. Let 
\begin{align*}
I'&=\textrm{ Free abelian group on all the }z_{k;i}, \textrm{ and}\\
I'_k&= \textrm{ Free abelian group on the set }\{z_{k;i}\}_{0 \leq i
    \leq s_k-1},\\
\textrm{so} & \textrm{ that } I'=\oplus_{1\leq k \leq n}I'_k.
\end{align*}
 The group $\pi_{s_k}$ is isomorphic to a subgroup of
$\Eva(\Q(\zeta_{s_k})/\Q)$, which cyclically permutes the set
$\{z_{k;i}\}_{\gcd(i,s_k)=1}$. Therefore, as a $\Z[\pi_{s_k}]$-module,
$I'_k$ contains at least one copy of $\Z[\pi_{s_k}]$. Since $\pi$
permutes the $z_{k;i}$, $I'$ is a $\Z[\pi]$-permutation module and
$I'$ decomposes into a direct sum of $\Z[\pi]$-permutation modules
with each summand having a $\Z$-basis on which $\pi$ acts
transitively. Therefore, recalling that $F_{m_s}(\Z[\pi_s])\cong
\Z[\zeta_{m_s}]$ and using \cite[Corollary 2.5]{Lenstra}, we have

\begin{equation} \label{exp}
l(I')^{\pi}\cong l(\oplus_{s \in O',
  m_s'|m_s}F_{m_s'}(I'))^{\pi}=l(\oplus_{s \in \bar{\Omega}}
\Z[\zeta_{m_s}]\oplus R)^{\pi},
\end{equation}

\noindent where $R$ just denotes the direct sum of the remaining summands.

\par Now define 
\[I''=I\oplus \Z\langle T \rangle,
\]
 where $\Z\langle T
\rangle $ denotes the (multiplicative) free abelian group on the set $T$. The group
$\pi$ acts trivially on $T$ and $\Z$ acts by taking powers. Note that if
$\zeta_{s_k}\in K$, then $\pi$ acts trivially on $J_{s_k}$, so by
construction, $l(J_{s_k})=l(J_{s_k})^{\pi}=l(x_{k;0})$. We
have the diagram of fields in Figure \ref{diamond2}, which follows from
Figure \ref{diamond1}.

\begin{figure}[h]
\caption{}\label{diamond2}
\begin{center}
  \setlength{\unitlength}{2.5cm}
  
 \begin{picture}(.3,1.4)(3,.6)
    
    \put(2.55,.6){$K'=l(I')^{\pi}$}
    \put(3.6,1.2){$K'(\zeta_e)=l(I')$}
    \put(.75,1.2){$K'l(I)^{\pi}=l(I)^{\pi}(T)=l(I'')^{\pi}$}
    \put(2.93, 1.8){$l(I,T)=l(I'')$}
   \put(2.25, 1.7){$m_e$}
   \put(3.5,.9){$m_e$}
   \put(3.5, 1.55){$d$}
   \put(2.16,.77){$d$}

    \qbezier(2.9,.8)(2.9,.8)(1.8,1.1)
     \qbezier(3.2,.8)(3.2,.8)(3.55, 1.1)
    \qbezier(1.8, 1.41)(1.8,1.41)(2.95,1.74)
    \qbezier (3.58, 1.45)(3.58,1.45)(3.13, 1.74)

\end{picture}

\end{center}
\end{figure}

By
\cite[Corollary 2.5]{Lenstra}, we have 
\begin{equation}\label{exp2}
l(I'')^{\pi}\cong l(\oplus_{s \in O',
  m_s'|m_s}F_{m_s'}(I \oplus \Z\langle T \rangle))^{\pi}.
\end{equation}
It follows from Proposition \ref{0Len} that $F_m(\Z \langle T \rangle)=0$ unless $m=1$, in
which case
\linebreak $F_m(\Z \langle T \rangle)=\Z^{|T|}$. Recalling that $I=\oplus_{s\in
  \Omega'} J_s$ and using line \eqref{J} from Section \ref{generalcase}, we have
\begin{equation}\label{exp3}
\oplus_{s \in O', m'_s|m_s}
F_{m'_s}(I \oplus \Z \langle T \rangle)=
\big(\oplus_{s \in
  O'}(\oplus_{m'_s\neq m_s}\Z[\zeta_{m'_s}]) \oplus
\mathfrak{a}_s \big) \oplus \Z^{|T|}.
\end{equation}

\par We would like to say something about $d$,
which we will do by working with $[l(I''):l(I')]$. We have $l(I_k')\subseteq l(J_{s_k} \oplus T_k)$,
where $T_k$ is the free abelian group on $\{x_{k;m_{s_k}},\ldots ,
x_{k;s_k-1} \}$. Since $I_i'$ and $I_j'$ are algebraically independent
(viewed as subsets of $l(I')$) for $i \neq j$, as are $J_{s_i}\oplus T_i$
and $J_{s_j}\oplus T_j$ (viewed as subsets of $l(I'')$), we have
\begin{equation*} 
[l(I''):l(I')]=\prod_{1\leq k \leq n}[l(J_{s_k} \oplus T_k):l(I_k')].
\end{equation*}

If $\zeta_{s_k} \notin K$, then
by assumption \eqref{test} at the beginning of Section \ref{secbelow},
we know that \newline
$m_{s_k}=[K(\zeta_s):K]$ is even. Let $d_k=[l(J_{s_k} \oplus
T_k):l(I_k')]$, and note that $d_k=1$ if
$\zeta_{s_k}\in K$.  
\subsection{Conditional Lower Bounds for $d_k$}
From \cite[Corollary 2.5]{Lenstra}, we have
\begin{equation} \label{cond}
l(I'_k)\cong
l(\oplus_{m'_{s_k}|m_{s_k}}F_{m'_{s_k}}(I'_k))\cong\big((\oplus_{m'_{s_k}|m_{s_k}}\Z[\zeta_{m'_{s_k}}])\oplus R_k\big),
\end{equation}
where $R_k$ denotes the remaining $\pi$-module summands
(\textit{cf}. line \eqref{exp}). 

Similarly, we have
\begin{equation} \label{condl}
l(J_{s_k} \oplus T_k)\cong
l\big(\oplus_{m'_{s_k}|m_{s_k}} F_{m'_{s_k}}(J_{s_k}\oplus T_k)\big)\cong l\big(\mathfrak{a}_{s_k}\oplus(\oplus_{m'_{s_k}|m_{s_k},m'_{s_k}\neq
m_{s_k}} \Z[\zeta_{m'_{s_k}}])\oplus \Z^{|T_k|}\big),
\end{equation}
 (\textit{cf}.
lines \eqref{exp2} and \eqref{exp3}). Note that in \eqref{cond} and
\eqref{condl}, both isomorphisms
respect $\pi$, the first isomorphism in each by \cite[Corollary 2.5]{Lenstra}, and the
second by \cite[Proposition 3.6]{Lenstra}
and Proposition \ref{0Len}. Set
\[
N=\mathfrak{a}_{s_k} \oplus (\oplus_{m'_{s_k}|m_{s_k},m'_{s_k}\neq
m_{s_k}} \Z[\zeta_{m'_{s_k}}])\oplus \Z^{|T_k|}, 
\textrm{ and } N'=(\oplus_{m'_{s_k}|m_{s_k}}\Z[\zeta_{m'_{s_k}}])\oplus R_k.
\]
We have that $l(N')$ is isomorphic (via an isomorphism that respects
$\pi$) to a subfield of $l(N)$, with
$[l(N):l(N')]=d_k$, and we assume without loss of
generality $l(N')\subseteq l(N)$. 
\par Recall that the cyclic group $\pi$ acts through its quotient $\pi_{s_k}$
on $l(N)$, acting on each
direct summand separately, and within each summand, acting by permuting
monomials. Set $w=m_{s_k}$, and let $M$ be a copy of $\Z[\zeta_{w}]$
in $N'$. Let $\sigma$ be a generator of
$\pi$. As $M$ is a free
$\Z[\zeta_w]$-module of rank 1, there exists an element $f \in l(N)$
so that 
\[
l(M)=l(f,f^{\sigma}, \ldots, f^{\sigma^{r-1}}),
\] where $r=\phi(w)$. The action of
$\sigma$ on $f$ is given by viewing $f$ as an element of
$l(N)$, on
which $\sigma$ acts as a field automorphism. Note that although
$\sigma$ is an automorphism of $l(N)$ that restricts naturally to
$l(N')$, when viewed as $\Z[\zeta_{w}]$-modules, the
action of $\Z$ on $N$ (taking powers of elements of $N$) is not compatible
with the action of $\Z$ on $N'$ (taking powers of elements of
$N'$). The element $\sigma^{\frac{w}{2}}$ acts
on $f$ as inversion since $\zeta_{w}^{w/2}=-1$ (recall that
$w$ is assumed to be even). We can
identify $l(N)$ with $l(y_0, \ldots ,
y_{s_k-1})$ for indeterminates $y_i$, so that we may write $f=\frac{g}{h}$ with $g, h \in B:=l[y_0^{\pm 1}, \ldots , y_{s_k-1}^{\pm
  1}]$, where $g$ and $h$ have no common factors (recall that
$B$ is a unique factorization
domain with unit group equal to the group of monomials). Setting
$a=\frac{w}{2}$, we have
\begin{equation*}  \frac{h}{g}=f^{-1}=f^{\sigma^{a}}
=\frac{g^{\sigma^{a}}}
 {h^{\sigma^{a}}},
\end{equation*}
so, up to units, $h=g^{\sigma^{a}}$. 
\par If $f \in B^*$, we have the following conditional results.

\subsubsection{The Case $f \in B^*$}\label{case1}

\textit{All results in this section are only valid under the
  assumption that $f \in B^*$}.\\
\par For the moment we additionally assume:
\begin{equation}\label{assume}
\textrm{If } s \in \Omega' \textrm{ and } \zeta_{m_s}\notin K,
\textrm{ then the element } s \textrm{ appears in } \Omega' \textrm{
  exactly once }.
\end{equation}
Here, $f$ must be a monomial in $l[y_0^{\pm 1}, \ldots ,
  y_{s_k-1}^{\pm 1}]$. If $f$ has coefficient $\alpha
  \in l$, then taking $f':=\frac{1}{\alpha}f$, we may replace $M$ by
  $M'$, where $M'$ is the $\Z[\pi]$-module generated by $f'$. So $M'$ is contained in $N$. 
\par
Write
$f'=\prod_{i=0}^{s_{k-1}}y_i^{a_i}$, where we may assume that $y_{0} , \ldots , y_{r-1}$
correspond to $\mathfrak{a}_{s_k}$ (recall $r=\phi(w)$). We claim that $a_i=0$ for all $r\leq i
\leq s_k-1$, and thus $M' \subseteq \mathfrak{a}_{s_k}$. To
see this, note that if $\tau$ is the $w$th cyclotomic polynomial in $\Z[\sigma]$, then
$f'^{\tau}=1$. The element $\tau$ acts on monomials in the
variables $y_{b_i},
\ldots, y_{b_{i+1}-1}$, where these groupings correspond to separate
summands of $N$ (so, for example, $b_0=0, b_1=r$). Therefore
$\tau$ acts trivially on each subproduct $y_{b_i}^{a_{b_i}}\cdots
y_{b_{i+1}-1}^{a_{b_{i+1}-1}}$. But this only holds if $i=0$
or if each exponent in the subproduct is zero since $y_{b_j}$, for
$b_j\geq r$,
satisfies $y_{b_j}^{\Phi_{w'}(\sigma)}=1$, for some $w' \neq w$, and
all cyclotomic polynomials are irreducible. We conclude that $M'$ is a
principal ideal contained in $\mathfrak{a}_{s_k}$. 

\begin{remark}
If $s_k$ occurred with multiplicity $n_k$ in $\Omega'$, then we would
only be assured that $M'$ is contained in $\mathfrak{a}_{s_k}^{n_k}$
(direct sum). 
\end{remark}

\begin{thm}\label{ckdk}
Let $A$ be an abelian group with elementary divisor decomposition \linebreak
$A=\oplus_{k \in \Omega} \Z/s_k\Z$, $K$ a field, and suppose
assumptions \eqref{test}, \eqref{test2}, and \eqref{assume} hold. For each $s_k$, let $c_k$ denote the minimum over all
principal ideals $\I_k \subseteq \mathfrak{a}_{s_k}$ of
$|\mathfrak{a}_{s_k}:\I_k|$. Then, using the notation above, $c_k \leq d_k$.
\end{thm}

\begin{proof}
We have
\[
c_k \leq [l(\mathfrak{a}_{s_k}):l(M')] 
\leq [l(N):l(N')]= d_k.
\]
\end{proof}

Suppose we remove assummption \eqref{assume}, so for a given $s_k$,
$J_{s_k}$ occurs in $I$ with multiplicity $n_k$, which may be greater
than 1. It follows from Remark \ref{course} that 
\[c_k^{n_k}\geq
\gon(l(J_{s_k}^{n_k}\oplus \Z \langle T_k \rangle^{n_k})^{\pi},
\]
and
from Theorem \ref{ckdk} that $c_k^{n_k}\leq d_k^{n_k}$. Informally
speaking, the likelihood of $c_k^{n_k}$ (and thus $d_k^{n_k}$) being close
to $\gon(l(J_{s_k}^{n_k}\oplus \Z\langle T_k\rangle^{n_k})^{\pi})$
decreases as $n_k$ gets larger. This is because the rationality of
$l(J_{s_k}^{n_k}\oplus \Z\langle T\rangle^{n_k})^{\pi}$ is equivalent to that of
$l(\oplus_{m'_{s_k}|m_{s_k}, m\neq
  m_{s_k}}\Z[\zeta_{m'_{s_k}}]^{n_k}\oplus\mathfrak{a}_s^{n_k} \oplus
\nolinebreak \Z^{n_k|T|})^{\pi}$, and the latter field is
rational if $\mathfrak{a}_{s_k}^{n_k}$ is a free $\Z[\pi_s]$-module, or,
equivalently \cite{kaplded}, if $\mathfrak{a}_{s_k}^{n_k}$ (ideal product inside $\Z[\zeta_{m_{s_k}}]$) is
principal. Moreover, by Section \ref{bounding}, if $\mathfrak{a}_s^{n_k}$ (direct sum) contains a free
$\Z[\zeta_{m_{s_k}}]$-module $X$ of index $d_X$, then
$l(J_{s_k}^{n_k}\oplus \Z\langle T_k\rangle^{n_k})^{\pi}$ will have
gonality at most $d_X$; we are 
guaranteed such an $X$ with $d_X \leq c_k^{n_k}$ since we can always
take $X=\I_k^{n_k}$ (direct sum) for a principal $\I_k \subseteq \mathfrak{a}_k$, a free $\Z[\zeta_{m_{s_k}}]$-module inside $\mathfrak{a}_s^{n_k}$. 

\par Fix a field $K$, and suppose the following holds:
\begin{align} \label{**}
 A=\Z/s\Z \textrm{ (so } d_1=d )\textrm{ is cyclic for an odd prime
power }s, \textrm{ and } \zeta_s \notin K. 
\end{align}
We then have the following conditional Theorem:
\begin{thm} \label{=}
Suppose \eqref{test}, \eqref{test2} and \eqref{assume} hold. Then $\gon(l(J_s)^{\pi})=c_1$. That is, $\gon(l(J_s)^{\pi}))$ is
equal to the minimum over all principal ideals $\I \subseteq
\mathfrak{a}_{s}$ of $|\mathfrak{a}_{s}:\I|$.
\end{thm}
\begin{proof}
By construction, $\gon(l(J_s)^{\pi})=d_1$. From Theorem \ref{ckdk},
$c_1 \leq d_1$. On the other hand, by Remark \ref{course},
$\gon(l(J_s)^{\pi})\leq c_1$.
\end{proof}

As corollary, to Theorem \ref{=}, we have

\begin{corollary}
Let $K$ be field that is finitely generated over its prime subfield,
and suppose that \eqref{test}, \eqref{test2}, and \eqref{assume} hold.
 Then for
almost all (Dirichlet density 1) primes $p$,
\[
\gon(J_p)\geq m_p+1=[K(\zeta_p):K]+1.
\]
\end{corollary}

\begin{proof}
Lenstra \cite[Corollary 7.6]{Lenstra} shows that the set of primes for which $K(\Z/p\Z)$ is not rational
has density one. The result now follows by Theorem \ref{=} and the
fact that the minimal index of any ideal $\I$ properly contained in
$\mathfrak{a}_p$ is $m_p+1$ (recall $\mathfrak{a}_p \subset
\Z[\zeta_{m_p}]$). 
\end{proof}

\subsubsection{Remarks in the Case that $f \notin B^*$}

We close by listing a few properties that $f$ must satisfy in the case
that $f \notin B^*$. 

\par Suppose that $f$ is not a monomial in $B$. If we continue to assume
\eqref{test} and \eqref{test2}, then, as previously noted, we can write
$f=\frac{g}{h}$, with $h=u^{-1}g^{\sigma^a}$, $u \in B^{*}$.  In this
case, $f$ must have the following properties:

\begin{prop}
The element $g \in B$ cannot be written as a sum of two or fewer terms
in $B$.
\end{prop}
\begin{proof}
Since $\sigma$ acts additively on monomials and
$f=\frac{ug}{g^{\sigma^a}}$ is assumed to not be a monomial, we may
assume that $g$ can be written as $A+B$ is a sum of two monomials. If $|A|=\prod
y_i^{a_i}$, define $|A|=\prod y_i^{|a_i|}$. Furthermore, define
$A^+=\prod_{a_i>0}y_i^{a_i}$, and $A^-=\prod_{a_i<0}y_i^{-a_i}$, so
that $|A|=A^+A^-$. Likewise, define $|B|, B^+, B^-$.
Then
\begin{equation}
f=u\frac{g}{h}=u\frac{|A||B|}{|A||B|}\cdot\frac{A+B}{A^{-1}+B^{-1}}=u\frac{(A^+)^2|B|+(B^+)^2|A|}{(A^-)^2|B|+(B^-)^2|A|}=u\frac{A^+B^+(A^+B^-+B^+A^-)}{A^-B^-(A^-B^++A^+B^-)},
\end{equation}
which is a monomial, a contradiction. 
\end{proof}

\begin{prop} \label{sigunit}
For every $b, 1 \leq b \leq w-1$, $\frac{g}{g^{\sigma^b}}\notin B^*$. 
\end{prop}
First, we have a lemma.
\begin{lemma} \label{nonontriv}
Let $M$ be the $\Z[\zeta_w]$-module generated by $f$. Then $M\cap B=1$.
\end{lemma}

\begin{proof}
Suppose $M$ contained a non-trivial element $v \in B$. Then $v$ generates a
$\Z[\zeta_w]$-module $<v>$ inside $M$, and $M/<v>$ is finite, meaning
that there exist $\delta_1, \ldots, \delta_n \in M$ such that ever
element of $M$ is of the form $\delta_iv'$, where $v' \in <v>$. This
is impossible, however, since powers of $f$ give elements of $M$ whose
numerator and denominator are both a product of arbitrarily many
irreducible elements of $B$. 
\end{proof}

\begin{proof}
Suppose we had $g^{\sigma^b}=vg$, for some $v \in B^*$, $1 \leq b \leq
w-1$. We have $f=\frac{ug}{g^{\sigma^a}}$, so 
\[
\frac{f}{f^{\sigma^b}}=\frac{v^{\sigma^a}u}{v u^{\sigma^b}}:=\Lambda.
\]
We cannot have $\Lambda=1$ since the lowest positive power of $\sigma$
fixing $f$ is $w$. But $\Lambda \neq 1$ contradicts Lemma
\ref{nonontriv}. 
\end{proof}

\begin{corollary}
The element $g$ cannot be an irreducible element of $B$. 
\end{corollary}

\begin{proof}
Suppose $g \in B$ were irreducible. As mentioned in Section
\ref{case1}, if $\tau$ is the $w$th cyclotomic polynomial in
$\Z[\sigma]$, then $f^{\tau}=1$. Since $B$ is a UFD, this would imply
that up to units, $g=g^{\sigma^b}$, for some $b$, $1 \leq b \leq w-1$,
contradicting Proposition \ref{sigunit}.
\end{proof}

\begin{corollary}
Every non-trivial element of $m \in M$ is of the form $\frac{e}{e'}$, where
$e$ and $e'$ are coprime elements of $B$, each a product of $a_m$
irreducible elements of $B$, with $a_m \geq 2$.
\end{corollary}

\begin{proof}
Any non-trivial element $m \in M$ generates a free
$\Z[\zeta_w]$-module, so as shown in the case $m=f$, $m$ must be of
the form $\frac{ug}{g^{\sigma^a}}$ with $u \in B^*$ and $g$ a product
of at least two irreducible elements in $B$. 
\end{proof}

\end{document}